\documentclass[11pt]{amsart}
\usepackage[margin=1in]{geometry}

\usepackage{amsthm,amsmath,amssymb,color}
\usepackage{adjustbox,bm,booktabs,caption,flexisym,float,longtable,multirow,pdflscape,bbm}
\usepackage{verbatim, appendix}
\usepackage[hidelinks]{hyperref}

\restylefloat{table}

\theoremstyle{definition}
\newtheorem{defi}{Definition}[section]
\newtheorem{ex}[defi]{Example}

\theoremstyle{plain}
\newtheorem{thm}[defi]{Theorem}
\newtheorem{lemma}[defi]{Lemma}
\newtheorem{cor}[defi]{Corollary}
\newtheorem{prop}[defi]{Proposition}

\newcommand{\lr}[1]{\lbrace #1 \rbrace}

\newcommand{\AGL}{{\rm AGL}}
\newcommand{\PGL}{{\rm PGL}}
\newcommand{\AGamL}{{\rm A \Gamma L}}
\newcommand{\PGamL}{{\rm P \Gamma L}}
\newcommand{\GF}{{\rm GF}}
\newcommand{\PG}{{\rm PG}}
\newcommand{\CT}{{\footnotesize \sf CT}}
\newcommand{\hd}{{\rm hd}}

\makeatletter
\def\imod#1{\allowbreak\mkern10mu({\operator@font mod}\,\,#1)} 
\def\@setcopyright{}                                           
\def\serieslogo@{}
\makeatother

\begin{document}

\author[C.~Amarra]{Carmen Amarra}
\address[C.~Amarra]{Institute of Mathematics, University of the Philippines Diliman, 1101 Quezon City, Philippines}
\email{mcamarra@math.upd.edu.ph}

\author[D.V.A.~Briones]{Dom Vito A.~Briones}
\address[D.V.A.~Briones]{Institute of Mathematics, University of the Philippines Diliman, 1101 Quezon City, Philippines}
\email{dabriones@up.edu.ph}

\author[M.J.C.~Loquias]{Manuel Joseph C.~Loquias}
\address[M.J.C.~Loquias]{Institute of Mathematics, University of the Philippines Diliman, 1101 Quezon City, Philippines}
\email{mjcloquias@math.upd.edu.ph}

\title{Multiple Contractions of Permutation Arrays}

\begin{abstract}{Given a permutation $\sigma$ on $n$ symbols $\lr{0, 1, \ldots, n-1}$ and an integer $1 \leq m \leq n-1$, the $m$th contraction of $\sigma$ is the permutation $\sigma^{\CT^m}$ on $n-m$ symbols obtained by deleting the symbols $n-1, n-2, \ldots, n-m$ from the cycle decomposition of $\sigma$. The Hamming distance $\hd(\sigma,\tau)$ between two permutations $\sigma$ and $\tau$ is the number of symbols $x$ such that $\sigma(x) \neq \tau(x)$.
In this paper we give a complete characterization of the effect of a single contraction on the Hamming distance between two permutations.
This allows us to obtain sufficient conditions for $\hd(\sigma,\tau)-\hd(\sigma^{\CT^m},\tau^{\CT^m})\leq 2m$.
}\end{abstract}

\subjclass[2010]{05A05, 94B25, 94B65}

\keywords{Permutation Arrays, Hamming distance, Contraction, AGL(1,q), PGL(2,q)}

\date{\today}

\maketitle

\section[Introduction]{Introduction}\label{sec1}

A \emph{permutation array} on $n$ symbols is a non-empty subset of the symmetric group $S_n$ on the $n$~symbols $\lr{0,1,\ldots,n-1}$. 
The \emph{Hamming distance} $\hd(\sigma,\tau)$ between two permutations $\sigma$ and $\tau$ in $S_n$ is the number of symbols $x$ such that $\sigma(x)\neq\tau(x)$, and the Hamming distance of a permutation array $P$ is the minimum among all Hamming distances between pairs of distinct permutations in $P$. For positive integers $n$ and $d$, $M(n,d)$ denotes the maximum size of a permutation array on $n$ symbols with Hamming distance~$d$.

The value of $M(n,d)$ for arbitrary $n$ and $d$ remains unsolved. 
Among the known general lower bounds for $M(n,d)$ are the Gilbert-Varshamov bound \cite{Deza_Max_Number_Permutations_1977} and $M(n,d)\geq n!/q^{d-2}$, where $q\geq n$ is a prime power \cite{Micheli2020}. 
For some special cases, exact values 
were obtained using combinatorial techniques \cite{Deza_Bounds_permutation_arrays_1978} or permutation group theory \cite{Deza_Max_Number_Permutations_1977}. 
For instance, results from \cite{Deza_Max_Number_Permutations_1977,Tent_Sharply_3_Transitive_2016} tell us that for prime powers $q$,
$M(q,q-1)=q(q-1)$ and $M(q+1,q-1)=q(q+1)(q-1)$.
In addition to this, methods for constructing permutation arrays that yield lower bounds for $M(n,d)$, for certain $n$ and $d$, have also been developed. These include the method of partitions and extensions \cite{Sudborough_Improving_MOLS_Bounds_2017}, permutation polynomials \cite{Chu_Codes_powerline_2004}, coset search algorithms \cite{Sudborough_Constructing_2018}, rational polynomials \cite{Bereg2021}, and clique search and the use of automorphisms or isometries \cite{Smith_table_permutation_codes_2012,Chu_Codes_powerline_2004,Janiszczak_permutation_codes_isometries_2015}.    

In \cite{Sudborough_Constructing_2018}, Bereg et al.~introduced the method of contraction to obtain a new permutation array from a given one. 
The \emph{contraction} of a permutation $\sigma$ in $S_n$ is the permutation $\sigma^{\CT}$ in $S_{n-1}$ obtained by deleting the symbol $n-1$ from the disjoint cycle notation of $\sigma$. Taking the contraction of each permutation in a permutation array $P$ yields the contraction $P^{\CT}$ of $P$. 
Bereg et al.~examined the effect of contraction on the Hamming distance of a given permutation array. They found that the Hamming distance between any two permutations can decrease by at most $3$ after one contraction, and they provided necessary and sufficient conditions in order for the decrease in Hamming distance to be equal to $3$. Moreover, they gave sufficient conditions for the Hamming distance between two permutations to decrease by $2$ after one contraction and to decrease by at most $4$ after two contractions.
They applied their results to the affine general linear group and the projective general linear group to yield lower bounds for $M(n,d)$ where $(n,d)$ is of
the form $(q,q-3)$ or $(q-1,q-3)$, where $q$ is a prime power such that $3\nmid (q-1)$, and of the form $(q-2,q-5)$, 
where $q$ is a prime power such that $q\equiv 2\imod{3}$ and $q\not\equiv 0,1\imod{5}$.
In \cite{Bereg2019}, lower bounds for $M(n,d)$ were obtained for the case where $(n,d)$ is either $(q,q-3)$ or $(q-1,q-3)$ for some prime power $q \equiv 1 \imod{3}$, using contraction graphs of certain permutation arrays. These are graphs whose vertices are the elements of the permutation array, with adjacency defined using conditions on the Hamming distance of two permutations and the Hamming distance of their contractions. Furthermore, Bereg et. al.~derived other lower bounds using contractions of permutations obtained from permutation rational functions \cite{Bereg2021}. These works suggest that contraction, used in combination with other computational methods, is a useful tool in deriving lower bounds for $M(n,d)$.

This paper extends the combinatorial work in \cite{Sudborough_Constructing_2018} on the effect of contraction on the Hamming distance between two permutations.
In particular, we identify necessary and sufficient conditions for $\hd(\sigma,\tau)-\hd(\sigma^{\CT},\tau^\CT)=d$ for each $d\in\{0,1,2,3\}$.
Moreoever, we consider the effect of more than two contractions on the Hamming distance of a permutation array. 
Our main result is Theorem \ref{theorem:lowerbound}, which gives a lower bound for the Hamming distance between the $m$th contractions
$\sigma^{\CT^m}$ and $\tau^{\CT^m}$ of two permutations $\sigma$ and $\tau$. 

\setcounter{section}{5}
\setcounter{defi}{1}
\begin{thm} 
	Let $\sigma, \tau \in S_n$ with $\hd(\sigma,\tau) = d$. Let $m \in \{1, \ldots, n-1\}$ and assume that the disjoint cycle decomposition of $\sigma^{-1}\tau$ has no factor of odd length $\ell$ where $3 \leq \ell \leq 2m+1$. Then $\hd\big(\sigma^{\CT^m},\tau^{\CT^m}\big) \geq d - 2m$. 
\end{thm}
\setcounter{section}{1}

Applying Theorem \ref{theorem:lowerbound} to the group $\PGL(2,q)$ for prime powers $q$ yields new lower bounds for $M(q-1,q-5)$, which we show in Table \ref{tab:Table 0.5}. Other lower bounds we obtain are improvements on the Gilbert-Varshamov bound \cite{Deza_Max_Number_Permutations_1977} and the bound $M(n,d)\geq n!/q^{d-2}$ in \cite{Micheli2020}, but not of some of the other current best lower bounds \cite{Chu_Codes_powerline_2004,Smith_table_permutation_codes_2012,Sudborough_Constructing_2018,Bereg2021}. 

The rest of this paper is organized as follows: In Section 2 we define the contraction of a permutation and give its basic properties. 
Next, we consider the contraction $P^{\CT}$ of a given permutation array $P$ in Section 3, and compare its size and Hamming distance with that of $P$. 
In Section 4, we determine necessary and sufficient conditions for each possible value of $\hd(\sigma,\tau) - \hd(\sigma^{\CT}, \tau^{\CT})$, for any two distinct permutations $\sigma$ and $\tau$. 
We push further in Section 5 and consider the $m$th contraction $P^{\CT^m}$ of a given permutation array $P$ for any $m$, 
and use the results of the previous sections to compare its size and Hamming distance with that of $P$. 
We conclude the paper in Section 6 by putting in context the lower bounds obtained in Section 5.

\section[Preliminaries]{Preliminaries}

We consider the symmetric group $S_n$ on the $n$ symbols $\lr{0, 1, \ldots, n-1}$.
We adopt the convention that compositions of permutations are computed from left to right.
That is, given $\sigma,\tau\in S_n$ and $x\in\lr{0,  \ldots, n-1}$, $\tau\sigma(x):=\sigma(\tau(x))$.
The support of a permutation $\sigma\in S_n$ is $\operatorname{supp}(\sigma)=\{x\in\{0,\ldots,n-1\}\,:\,\sigma(x)\neq x\}$.

A \emph{permutation array} on $n$ elements is a non-empty subset of $S_n$.
The \emph{Hamming distance} $\hd(\sigma,\tau)$ between permutations $\sigma,\tau \in S_n$ is
    \[ \hd(\sigma,\tau) = \vert \lr{x \in \lr{0, \ldots, n-1} \,:\,\sigma(x)\neq \tau(x)}\vert. \]
The Hamming distance $\hd(P)$ of a permutation array $P$ is given by     \[ \hd(P) = \min \lr{\hd(\sigma,\tau) \,:\, \sigma,\tau \in P,\; \sigma \neq \tau}. \]
We denote by $M(n,d)$ the maximum size of a permutation array on $n$ symbols with Hamming distance $d$.

Permutation arrays with maximum possible size can be obtained from groups which act on sets in a particular way. A group $G$ acting on a set $X$ is said to be \emph{sharply $k$-transitive} (for some positive integer $k \leq \vert X\vert $) if it has a single orbit on the set of all ordered $k$-tuples of distinct elements of $X$, and if only the identity element fixes any such $k$-tuple. It was shown in \cite{Deza_Max_Number_Permutations_1977} that a group acting sharply $k$-transitively on a set of $n$ elements has Hamming distance $d = n-k+1$ and has the maximum possible size $M(n,d)$. The sharply $k$-transitive groups are all known: they exist only for $k \leq 5$, and the only examples are a few infinite families and some sporadics such as the Mathieu groups \cite{dixon_permutation_1996}.  
In particular, apart from a finite number of examples, the only sharply $2$-transitive groups are the one-dimensional affine groups $\AGL(1,q)$ in their natural action on the field $\GF(q)$ of order $q$. Hence, for any prime power $q$,
    \[ M(q,q-1) = q(q-1). \]
There are two infinite families of sharply $3$-transitive groups, all of which arise from finite fields: one is $\PGL(2,q)$ in its natural action on the points of the projective line $\PG(1,q) = \GF(q) \cup \{\infty\}$. For odd prime powers $q$ there is a second family of sharply $3$-transitive groups: 
a group also of order $q(q+1)(q-1)$ 
but is not permutation isomorphic to $\PGL(2,q)$ acting on $\PG(1,q)$ \cite{dixon_permutation_1996}.  
Thus, for any prime power $q$,
    \[ M(q+1,q-1) = q(q+1)(q-1). \]
    
The \emph{contraction} of $\sigma \in S_n$ is the permutation $\sigma^{\CT} \in S_{n-1}$ defined by
    \[ \sigma^{\CT}(x) = \begin{cases} 
            \sigma(n-1) &\text{if } x = \sigma^{-1}(n-1), \\
            \sigma(x) &\text{otherwise.}
            \end{cases} \]
If $\sigma$ is a cycle that fixes $n-1$, then $\sigma^{\CT} = \sigma$.
On the other hand, if $\sigma$ is the cycle $(a_1 \ a_2 \ \cdots \ a_{k-1} \ n-1)$ for some $a_1, a_2, \ldots, a_{k-1} \in \{0, \ldots, n-2\}$ 
then $\sigma^{\CT}$ is the cycle $(a_1 \ a_2 \ \cdots \ a_{k-1})$ in $S_{n-1}$ obtained by deleting $n-1$ from $\sigma$. 
More generally, if $\sigma$ has disjoint cycle decomposition $\sigma_1 \sigma_2 \cdots \sigma_k$ then $\sigma^{\CT} = \sigma_1^{\CT} \sigma_2^{\CT} \cdots \sigma_k^{\CT}$. 
Observe that there is nothing special about the symbol $n-1$ insofar as we may contract with respect to any of the symbols in $\{0,1,\ldots,n-1\}$ (cf.~\cite{Bereg2019}).
For the purposes of this paper, we shall always contract a permutation array $P\subseteq S_n$ with respect to the symbol $n-1$.

Lemma \ref{lem:inverse-CT} can be easily deduced from the definition above.

\begin{lemma} \label{lem:inverse-CT}
For any $\sigma \in S_n$, $\left(\sigma^{\CT}\right)^{-1} = \left(\sigma^{-1}\right)^{\CT}$. 
\end{lemma}
\begin{proof}
It suffices to consider cycles $\sigma$ which does not fix $n-1$, say, $\sigma = (a_1 \ a_2 \ \cdots \ a_{k-1} \ n-1)$ for some $a_1, a_2, \ldots, a_{k-1} \in \{0,\ldots, n-2\}$.
We have
    \[ \left(\sigma^{\CT}\right)^{-1} = (a_1 \ a_{k-1} \ a_{k-2} \ \cdots \ a_2) = \left(\sigma^{-1}\right)^{\CT}.\qedhere\]
\end{proof}

For any $\sigma \in S_n$ and integer $m$ with $1 \leq m \leq n-1$,
    \[ \sigma^{\CT^m} := \left(\sigma^{\CT^{m-1}}\right)^{\CT} \]
and for any $P \subseteq S_n$,
    \[ P^{\CT^m} := \left\{\sigma^{\CT^m}\,:\,\sigma \in P \right\}. \]

Bereg et al.~gave a bound on the change in Hamming distance between two permutations after contraction, and in addition provided necessary and sufficient conditions for some cases.
These can be found in the observations in the discussion leading to the statement of \cite[Lemma 4]{Sudborough_Constructing_2018}, and in the statement itself.
We collect all of these in Lemma \ref{lemma:hd-bounds}.

\begin{lemma} \label{lemma:hd-bounds} \cite{Sudborough_Constructing_2018}
The following hold for any $\sigma,\tau\in S_n$.
\begin{enumerate}
    \item $\hd(\sigma,\tau)-\hd\left(\sigma^{\CT},\tau^{\CT}\right)\leq 3$
    \item The equation $\hd(\sigma,\tau) - \hd\left(\sigma^{\CT},\tau^{\CT}\right) = 3$ holds
    if and only if the cycle \[\rho = \big(n-1, \ \sigma(n-1), \ \tau(n-1)\big)\] is a factor in the disjoint cycle decomposition of $\sigma^{-1}\tau$. Furthermore, $\left(\sigma^{\CT}\right)^{-1}\tau^{\CT} = \sigma^{-1}\tau\rho^{-1}$.
    \item If $\hd(\sigma, \tau) - \hd(\sigma^{\CT} , \tau^{\CT} ) = 2$ then
        \begin{enumerate}
        \item[a.] $\sigma^{-1}\tau(n-1) \neq n-1$, and 
        \item[b.] if $\pi$ is the cycle containing $n-1$ in the disjoint cycle decomposition of $\sigma^{-1}\tau$, then the disjoint cycle decomposition of $(\sigma^{\CT} )^{-1}\tau^{\CT}$  is the same as that of $\sigma^{-1}\tau$ except that $\pi$ is replaced by the cycle obtained by removing from $\pi$ the symbols $n-1$ and one of $\pi(n-1)$ or $\pi^{-1}(n-1)$.
        \end{enumerate}
        \end{enumerate}

\end{lemma}

The following theorem is a consequence of Lemma \ref{lemma:hd-bounds}.

\begin{thm} \cite[Theorem 3]{Sudborough_Constructing_2018} \label{theorem:hd-P}
Let $P \subseteq S_n$ be a permutation array with Hamming distance $d$. Then the following hold:
    \begin{enumerate}
    \item $\hd\Big(P^{\CT^2}\Big) \geq d-6$;
    \item if for any $\sigma, \tau \in P$ the disjoint cycle decomposition of $\sigma^{-1}\tau$ has no cycles of length $3$ or $5$, then $\hd\Big(P^{\CT^2}\Big) \geq d-4$.
    \end{enumerate}
\end{thm}

Applying Theorem \ref{theorem:hd-P} to the permutation arrays $P^{\CT^2}$ for $P = \AGL(1,q)$, and to $P^{\CT}$ for $P = \PGL(2,q)$, yields the following corollary.

\begin{cor} \cite[Corollary 4, Corollary 5, and Theorem 4]{Sudborough_Constructing_2018}
Let $q$ be a prime power. 
    \begin{enumerate}
    \item If $q \not\equiv 1\imod{3}$ then
        \[ M(q-1, \, q-3) \geq q^2-1 \]
    and
        \[ M(q, \, q-3) \geq (q+1)q(q-1). \]
    \item If $q \equiv 2 \imod{3}$ and $q \not\equiv 0, 1 \imod{5}$ then
        \[ M(q-2, \, q-5) \geq q(q-1). \]
\end{enumerate} 
\end{cor}

\section[Injectivity of Contraction]{The order and Hamming distance of \texorpdfstring{$P^{\CT}$}{}}

Clearly $\vert P^{\CT}\vert \leq \vert P\vert$; we have $\vert P^{\CT}\vert < \vert P\vert$ exactly when there are two distinct permutations $\sigma$ and $\tau$ in $P$ with $\sigma^{\CT} = \tau^{\CT}$. Theorem \ref{theorem:order} gives an easy sufficient condition in order for the $m$th contraction of $P$ to have the same order as $P$.

\begin{thm} \label{theorem:order}
Let $P \subseteq S_n$ be a permutation array with Hamming distance $ d$. Then $\vert P^{\CT^m}\vert = \vert P\vert$ for any positive integer $m < d/3 $.
\end{thm}
\begin{proof}
By Lemma \ref{lemma:hd-bounds}.1, $\hd\big(\sigma^{\CT}, \tau^{\CT}\big) \geq \hd(\sigma, \tau) - 3$ for any $\sigma, \tau \in P$. As such, $\hd(\sigma^{\CT^m}, \tau^{\CT^m}) \geq \hd(\sigma,\tau) - 3m \geq d - 3m > 0$ for any integer $m < d/3$. Hence $\vert P^{\CT^m}\vert = \vert P\vert$ for all such $m$, as required.
\end{proof}

Next we consider the effect of the contraction operation on the Hamming distance of a permutation array. The Hamming distance may increase or decrease, as illustrated in the next example.

\begin{ex}
Let $P_1 = \lbrace (0), (1 \ 4), (0 \ 1 \ 2), (1 \ 2 \ 3) \rbrace \subseteq S_5$. We obtain $P_1^{\CT} = \lbrace (0), (0 \ 1 \ 2), (1 \ 2 \ 3) \rbrace \subseteq S_4$, and
$\hd\big(P_1^{\CT}\big) = 3>2=\hd(P_1)$.

Now consider $P_2 = \lbrace (0 \ 1 \ 2), (1 \ 3 \ 2) \rbrace \subseteq S_4$. Then $P_2^{\CT} = \{ (0 \ 1 \ 2), (1 \ 2) \}\subseteq S_3$. Notice in this case that $\hd\big(P_2^{\CT}\big) = 2<3=\hd(P_2)$.
\end{ex}

Theorem \ref{theorem:hd-increase} gives a necessary and sufficient condition for the Hamming distance of a permutation array to increase after contraction.

\begin{thm} \label{theorem:hd-increase}
Let $P$ be a permutation array with $\hd(P) = d$ and $\vert P^{\CT}\vert \geq 2$. Then $\hd\big(P^{\CT}\big) > \hd(P)$ if and only if the following conditions hold:
    \begin{enumerate}
    \item \label{condition2} any $\sigma, \tau \in P$ with $\hd(\sigma,\tau) = d$ satisfy $\sigma^{\CT} = \tau^{\CT}$; and   
    \item \label{condition3} any $\sigma,\tau \in P$ with $\hd(\sigma, \tau) > d$ satisfy $\hd\big(\sigma^{\CT}, \tau^{\CT}\big) > d$ or $\sigma^{\CT} = \tau^{\CT}$.
    \end{enumerate}
\end{thm}

\begin{proof}
Suppose conditions \eqref{condition2} and \eqref{condition3} hold. Let $\pi$ and $\rho$ be distinct permutations in $P^{\CT}$, and let $\sigma, \tau \in P$ such that $\sigma^{\CT} = \pi$ and $\tau^{\CT} = \rho$. Since $\hd(P) = d$, we then have $\hd(\sigma,\tau) \geq d$. 
However, $\hd(\sigma,\tau) \neq d$, because otherwise, $\pi=\rho$ by \eqref{condition2}.
Thus $\hd(\sigma,\tau) > d$ and we obtain from \eqref{condition3} that $\hd(\pi,\rho) > d$ by \eqref{condition3} .
Since $\pi$ and $\rho$ are arbitrary it follows that $\hd\big(P^{\CT}\big) > d$.

For the other direction, assume first that \eqref{condition2} does not hold.
Then there are permutations $\sigma,\tau \in P$ with $\hd(\sigma,\tau) = d$ and $\sigma^{\CT} \neq \tau^{\CT}$. 
This implies that
    \[ \hd\big(P^{\CT}\big) \leq \hd\big(\sigma^{\CT},\tau^{\CT}\big) \leq \hd(\sigma,\tau) = d = \hd(P).\]
On the other hand, if \eqref{condition3} does not hold, then there are permutations $\sigma,\tau \in P$ with $\hd(\sigma,\tau) > d$
such that $\sigma^{\CT}\neq \tau^{\CT}$ and $\hd(\sigma^{\CT},\tau^{\CT}) \leq d$. Again we have
    \[ \hd\big(P^{\CT}\big) \leq \hd\big(\sigma^{\CT},\tau^{\CT}\big) \leq d = \hd(P),\]
which completes the proof.
\end{proof}

It follows from Theorem \ref{theorem:hd-increase} that $\hd\big(P^{\CT}\big) \leq \hd(P)$ whenever $\vert P^{\CT} \vert = \vert P \vert$, since in this case $\sigma^{\CT} \neq \tau^{\CT}$ for any distinct $\sigma, \tau \in P$, and in particular for those with $\hd(\sigma,\tau) = d$. This observation together with Theorem \ref{theorem:order} yields the following result.

\begin{cor}
If $\hd(P) \geq 4$ then $\hd\big(P^{\CT}\big) \leq \hd(P)$.
\end{cor}

\begin{proof}
If $d := \hd(P) \geq 4$ then $d/3> 1$,
so that $\vert P^{\CT}\big\vert = \vert P \vert$ by Theorem \ref{theorem:order}. Thus, by the comments above, $\hd\big(P^{\CT}\big) \leq \hd(P)$.
\end{proof}

\section{Hamming distance of two permutations after contraction}

In this section we investigate the effect of the contraction operation on the Hamming distance of two permutations. The main results in this section are Propositions \ref{prop:hd-decrease} and \ref{prop:decomposition}, both of which will be used in the proof of the main result in Section 5. Proposition \ref{prop:hd-decrease} gives necessary and sufficient conditions for each possible change in Hamming distance, while Proposition \ref{prop:decomposition} describes the disjoint cycle decomposition of $\sigma^{-1}\tau$ for permutations $\sigma$ and $\tau$ which contract to different permutations.

Let $\sigma \in S_n$ be a cycle which does not fix $n-1$. Recall from the definition of the contraction operation that for any $x \in \{0, 1, \ldots, n-2\}$,
    \[ \sigma^{\CT}(x) = \left\{
        \begin{aligned}
        &\sigma(x) &&\text{if } x \neq \sigma^{-1}(n-1), \\
        &\sigma(n-1) &&\text{if } x = \sigma^{-1}(n-1).
        \end{aligned} \right. \]
Hence $\sigma^{\CT}(x) = \sigma(x)$ if and only if $x \neq \sigma^{-1}(n-1)$. It follows that to determine the Hamming distance after contraction of two permutations $\sigma$ and $\tau$ in $S_n$, we only need to examine the images of $n-1$, $\sigma^{-1}(n-1)$, and $\tau^{-1}(n-1)$.
Let
    \begin{equation} \label{eq:delta}
    \delta(\sigma,\tau) = \big\vert\big\{ x \in \big\{n-1, \, \sigma^{-1}(n-1), \, \tau^{-1}(n-1)\big\} \,:\,\sigma(x) \neq \tau(x) \big\}\big\vert
    \end{equation}
and
    \begin{equation} \label{eq:delta-CT}
    \delta^{\CT}(\sigma,\tau) = \big\vert\big\{ x \in \big\{\sigma^{-1}(n-1), \, \tau^{-1}(n-1)\big\} \,:\,\sigma^{\CT}(x) \neq \tau^{\CT}(x) \big\}\big\vert.
    \end{equation}

\begin{lemma} \label{lemma:delta}
Let $\sigma, \tau \in S_n$, and let $\delta(\sigma,\tau)$ and $\delta^{\CT}(\sigma,\tau)$ be as in \eqref{eq:delta} and \eqref{eq:delta-CT}, respectively. Then
    \[ \hd(\sigma,\tau) - \hd\big(\sigma^{\CT},\tau^{\CT}\big) = \delta(\sigma,\tau) - \delta^{\CT}(\sigma,\tau). \]
\end{lemma}

\begin{proof}
Let $\gamma = \big\vert \big\{ x \notin \big\{n-1, \, \sigma^{-1}(n-1), \, \tau^{-1}(n-1)\big\}\,:\,\sigma(x) \neq \tau(x) \big\}\big\vert $. Then $\hd(\sigma,\tau) = \gamma + \delta(\sigma,\tau)$. Since we also have $\sigma^{\CT}(x) = \sigma(x)$ and $\tau^{\CT}(x) = \tau(x)$ for any $x \notin \big\{n-1, \, \sigma^{-1}(n-1), \, \tau^{-1}(n-1)\big\}$, we have $\hd\big(\sigma^{\CT},\tau^{\CT}\big) = \gamma + \delta^{\CT}(\sigma,\tau)$. The result follows immediately.
\end{proof}
  
 \medskip 
  
Recall from Lemma \ref{lemma:hd-bounds} that for any distinct permutations $\sigma$ and $\tau$,
    \[ \hd(\sigma,\tau) - 3 \leq \hd\big(\sigma^{\CT},\tau^{\CT}\big) \leq \hd(\sigma,\tau). \] The following proposition describes when each possible decrease in Hamming distance occurs.

\begin{prop} \label{prop:hd-decrease}
Let $\sigma, \tau \in S_n$, and let $i = \sigma^{-1}(n-1)$, $j = \tau^{-1}(n-1)$, $a = \tau(i)$, $b = \sigma(j)$, $c = \sigma(n-1)$, and $d = \tau(n-1)$. Let $\delta(\sigma,\tau)$ and $\delta^{\CT}(\sigma,\tau)$ be as in \eqref{eq:delta} and \eqref{eq:delta-CT}, respectively. Then $\hd\big(\sigma,\tau\big)-\hd\big(\sigma^{\CT},\tau^{\CT}\big) =  \Delta\hd\big(\sigma,\tau\big)$ 
holds as given in Table \ref{table:hd-conditions}.
\end{prop}

\begin{table}[H]
\begin{center}
\begin{tabular}{rc|l}
\hline  
& $\Delta \hd(\sigma,\tau)$ & Conditions \\
\hline \hline
\textsf{\small 1} & $0$ & $i$, $j$, $n-1$ pairwise distinct; $a \neq b$ and $c = d$ \\
\cline{3-3}
\textsf{\small 2} &  & $i$, $j$, $n-1$ pairwise distinct; $a = b$ and $c = d$ \\
\cline{3-3}
\textsf{\small 3} &  & $i = j \neq n-1$, $a = b = n-1$, $c \neq d$ \\ 
\cline{3-3}
\textsf{\small 4} &  & $i = j \neq n-1$; $a = b = n-1$, $c = d$ \\
\cline{3-3}
\textsf{\small 5} &  & $i = j = n-1$; $a = b = c = d = n-1$ \\
\hline
\textsf{\small 6} & $1$ & $i$, $j$, $n-1$ pairwise distinct; $a$, $b$, $c$, $d$ pairwise distinct \\
\cline{3-3}
\textsf{\small 7} &  & $i$, $j$, $n-1$ pairwise distinct; $a = b$, $c \neq d$ \\
\cline{3-3}
\textsf{\small 8} &  & $i = n-1 \neq j$; $a = d$, $b$, $c = n-1$ pairwise distinct \\
\cline{3-3}  
\textsf{\small 9} &  & $j = n-1 \neq i$; $a$, $b = c$, $d = n-1$ pairwise distinct \\
\hline
\textsf{\small 10} & $2$ & $i$, $j$, $n-1$ pairwise distinct; $a = c$, $b \neq d$ \\
\cline{3-3} 
\textsf{\small 11} &  &  $i$, $j$, $n-1$ pairwise distinct; $a \neq c$, $b = d$ \\
\cline{3-3}    
\textsf{\small 12} &  & $i = n-1 \neq j$; $a = b = d \neq c = n-1$ \\
\cline{3-3}
\textsf{\small 13} &  & $j = n-1 \neq i$; $a = b = c \neq d = n-1$ \\
\hline
\textsf{\small 14} & $3$ & $i$, $j$, $n-1$ pairwise distinct; $a = c$, $b = d$ \\
\hline
\end{tabular}
\caption{Necessary and sufficient conditions for each possible value of $\Delta\hd(\sigma,\tau) = \hd(\sigma,\tau) - \hd\big(\sigma^{\CT},\tau^{\CT}\big)$}
\label{table:hd-conditions}
\end{center}
\end{table}

\begin{proof}
The images of $i$, $j$, and $n-1$ under $\sigma$, $\tau$, $\sigma^{\CT}$, and $\tau^{\CT}$ are given in the following tables:
    \begin{center}
    \begin{minipage}[t]{0.4\textwidth}
    \begin{tabular}{l|lll}
    \hline
    $x$ & $i$ & $j$ & $n-1$ \\
    \hline\hline
    $\sigma(x)$ & $n-1$ & $b$ & $c$ \\
    $\tau(x)$ & $a$ & $n-1$ & $d$ \\
    \hline
    \end{tabular}
    \end{minipage} \hspace{1cm}
    \begin{minipage}[t]{0.25\textwidth}
    \begin{tabular}{l|ll}
    \hline
    $x$ & $i$ & $j$ \\
    \hline\hline
    $\sigma^{\CT}(x)$ & $c$ & $b$ \\
    $\tau^{\CT}(x)$ & $a$ & $d$ \\
    \hline
    \end{tabular}
    \end{minipage}
    \end{center}
The values of $\delta(\sigma,\tau)$ and $\delta^{\CT}(\sigma,\tau)$ are easily obtained from the tables by counting the number of columns having different entries in each row. We divide into four cases.

\emph{Case 1.} Assume that $i$, $j$, and $n - 1$ are pairwise distinct. Since $\sigma$ and $\tau$ are bijections, it follows that $a \neq d$, $b \neq c$, and $a$, $b$, $c$, and $d$ are each distinct from $n - 1$. Furthermore, either $a$, $b$, $c$, and $d$ are
pairwise distinct, or at least one of the following holds: $a = b$, $a = c$, $b = d$, $c = d$. The only possible relationships among $a$, $b$, $c$, and $d$, together with the corresponding values of $\delta(\sigma,\tau)$ and $\delta^{\CT}(\sigma,\tau)$, are given in Table~\ref{tab:imp2}. These correspond to the entries in lines 1, 2, 6, 7, 10, 11, and 14 of Table \ref{table:hd-conditions}.

\begin{table}[ht]
\begin{center}
\begin{tabular}{lcc}
\hline
Conditions on $a,b,c,d$ & $\delta(\sigma,\tau)$ & $\delta^{\CT}(\sigma,\tau)$ \\
\hline \hline
$a$, $b$, $c$, $d$ pairwise distinct & $3$ & $2$ \\
$a = b$, $c$, $d$ pairwise distinct & $3$ & $2$ \\
$a = b \neq c = d$ & $2$ & $2$ \\
$a = c$, $b$, $d$ pairwise distinct & $3$ & $1$ \\          
$a = c \neq b = d$ & $3$ & $0$ \\
$a$, $b = d$, $c$ pairwise distinct & $3$ & $1$ \\
$a$, $b$, $c = d$ pairwise distinct & $2$ & $2$ \\
\hline
\end{tabular}
\caption {Values of $\delta$ and $\delta^{\CT}$ for Case 1 of the proof of Proposition \ref{prop:hd-decrease}}
\label{tab:imp2}
\end{center}
\end{table}

\emph{Case 2.} Assume that either $i = n-1 \neq j$ or $j = n-1 \neq i$. If $i = n-1 \neq j$ then $a = d \neq n-1$ and $b \neq c = n-1$. Similarly, if $j = n-1 \neq i$ then $b = c \neq n-1$ and $a\neq d = n-1$. The only possible relationships among $a$, $b$, $c$, and $d$, together with the corresponding values of $\delta(\sigma,\tau)$ and $\delta^{\CT}(\sigma,\tau)$, are given in Table \ref{tab:imp3}. These correspond to lines 8, 9, 12, and 13 of Table \ref{table:hd-conditions}.

\begin{table}[ht]
\begin{center}
\begin{tabular}{lcc}
\hline
Conditions on $a$, $b$, $c$, $d$ & $\delta(\sigma,\tau)$ & $\delta^{\CT}(\sigma,\tau)$  \\
\hline \hline  
$a = d$, $b$, $c = n-1$ pairwise distinct & $2$ & $1$ \\
$a = b = d \neq c = n-1$ & $2$ & $0$ \\
$a$, $b = c$, $d =  n-1$ pairwise distinct & $2$ & $1$ \\
$a = b = c \neq d = n-1$ & $2$ & $0$\\
\hline
\end{tabular}
\caption{Values of $\delta$ and $\delta^{\CT}$ for Case 2 of the proof of Proposition \ref{prop:hd-decrease}}
\label{tab:imp3}
\end{center}
\end{table}

\emph{Case 3.} Assume that $i = j \neq n-1$. Then $a = n-1 = b$, $c \neq n-1$, and $d \neq n-1$. It follows that $a = b \neq c$, $a = b \neq d$, and either $c = d$ or $c \neq d$. The only possible relationships among $a$, $b$, $c$, and $d$ and the corresponding values of $\delta(\sigma,\tau)$ and $\delta^{\CT}(\sigma,\tau)$ are shown in Table \ref{tab:imp4}. These correspond to lines 3 and 4 of Table \ref{table:hd-conditions}.

\begin{table}[ht]
\begin{center}
\begin{tabular}{lcc}
\hline
Conditions on $a$, $b$, $c$, $d$ & $\delta(\sigma,\tau)$ & $\delta^{\CT}(\sigma,\tau)$  \\
\hline\hline  
$a = b$, $c$, $d$ pairwise distinct & $1$ & $1$ \\
$a = b \neq c = d$ & $0$ & $0$ \\
\hline
\end{tabular}
\caption{Values of $\delta$ and $\delta^{\CT}$ for Case 3 of the proof of Proposition \ref{prop:hd-decrease}}
\label{tab:imp4}
\end{center}
\end{table}

\emph{Case 4.} Assume that $i = j = n-1$. Then $a = b = c = d = n-1$ and so $\delta(\sigma,\tau) = \delta^{\CT}(\sigma,\tau) = 0$. This yields line 5 of Table \ref{table:hd-conditions}.

Statements 1--4 of the proposition follow immediately from Table \ref{table:hd-conditions}. This completes the proof.
\end{proof}

The proof of Proposition \ref{prop:decomposition} relies on the next lemma, which tells us that the actions of the permutations $\sigma^{-1}\tau$ and $\big(\sigma^{\CT}\big)^{-1}\tau^{\CT}$ are mostly the same. Recall from Lemma \ref{lem:inverse-CT} that the operations of contraction and inversion commute, and that for any $\sigma \in S_n$, $\sigma^{\CT}(x) = \sigma(x)$ whenever $x\neq \sigma^{-1}(n-1)$.

\begin{lemma} \label{lemma:images}
Let $\sigma ,\tau \in S_n$ and $x\in\{0,1,\ldots,n-2\}$. If $x \notin \big\{ \sigma(n - 1), \tau^{-1}\sigma(n-1) \big\}$, then 
$\big(\sigma^{\CT}\big)^{-1}\tau^{\CT}(x) = \sigma^{-1}\tau(x)$.
\end{lemma}

\begin{proof}
It follows from Lemma \ref{lem:inverse-CT} and from $x \neq \sigma(n-1)$ that
\[ \sigma^{-1}(x) = \big(\sigma^{-1}\big)^{\CT}(x) = \big(\sigma^{\CT}\big)^{-1}(x). \]
Note that $y := \sigma^{-1}(x) \neq \tau^{-1}(n-1)$. Thus
    \[ \sigma^{-1}\tau(x) = \tau(y) = \tau^{\CT}(y) = \sigma^{-1}\tau^{\CT}(x) = \big(\sigma^{\CT}\big)^{-1}\tau^{\CT}(x).\qedhere\]
\end{proof}

We denote the length of a cycle $\rho$ by $\vert \rho\vert $. Following \cite{Sudborough_Constructing_2018}, we define the length of the identity permutation to be zero. Likewise, we will say that a symbol $x\in \{0, 1, \ldots, n-1\}$ \emph{belongs to a cycle of zero length} in the disjoint cycle decomposition of $\sigma \in S_n$ whenever $\sigma(x) = x$.

\begin{prop}\label{prop:decomposition}
Let $\sigma, \tau \in S_n$ such that $\sigma^{\CT} \neq \tau^ {\CT}$. Let $\omega$ and $\pi$ be the cycles in the disjoint cycle decomposition of $\sigma^{-1}\tau$ which contain $n-1$ and $\sigma(n-1)$, respectively. Suppose that
    \[ \sigma^{-1}\tau = 
    \begin{cases}
    \rho\,\omega\,\pi & \text{if } \omega\neq\pi, \\
    \rho\, \omega& \text{if } \omega=\pi,
    \end{cases}\]
where $\rho \in S_n$ is disjoint from $\omega$ and $\pi$. Then
    \[ \big(\sigma^{\CT}\big)^{-1}\tau^{ CT} = \rho\,\chi \]
for some $\chi \in S_{n-1}$ disjoint from $\rho$, and the following hold.
    \begin{enumerate}
    \item If $\hd\big(\sigma^{\CT},\tau^{\CT}\big) = \hd(\sigma, \tau)$, then either $\omega = (0)$ or $\pi = (0)$, and $\chi$ is a cycle with $\vert \chi\vert  = \vert \omega\,\pi\vert $.
    \item Suppose that $\hd\big(\sigma^{\CT},\tau^{\CT}\big) = \hd(\sigma,\tau) - 1$. 
        \begin{enumerate}
        \item If $\omega = \pi$ then either $\chi$ is a cycle with $\vert \chi\vert  = \vert \omega\vert  - 1$, or $\chi$ is a product of two disjoint cycles $\omega'$ and $\pi'$ with $\vert \omega'\vert  + \vert \pi'\vert  = \vert \omega\vert  - 1$.
        \item If $\omega \neq \pi$ then $\chi$ is a cycle with $\vert \chi\vert  = \vert \omega\vert  + \vert \pi\vert  - 1$.
        \end{enumerate}
       \item If $\hd\big(\sigma^{\CT},\tau^{\CT}\big) = \hd(\sigma,\tau) - 2$, then $\omega = \pi$ and $\chi$ is a cycle with $\vert \chi\vert  = \vert \omega\vert  - 2$.
       \item If $\hd\big(\sigma^{\CT},\tau^{\CT}\big) = \hd(\sigma,\tau) - 3$, then $\omega = \pi$, $\vert \omega\vert  = 3$, and $\chi = (0)$.
    \end{enumerate}
\end{prop}

\begin{proof}
Let $i = \sigma^{-1}(n-1)$, $j = \tau^{-1}(n-1)$, $a = \tau(i)$, $b = \sigma(j)$, $c = \sigma(n-1)$, and $d = \tau(n-1)$. Since $\sigma^{-1}\tau(b) = n-1$, the disjoint cycle decomposition of $\sigma^{-1}\tau$ contains a cycle $\omega := (b, \, n-1, \, a, \mathbf{x})$, where $\mathbf{x}$ is either empty or a sequence of distinct elements of $\{0, 1, \ldots, n-1\} \setminus \{a, b, n-1\}$. Likewise $\sigma^{-1}\tau(c) = d$, so the disjoint cycle decomposition of $\sigma^{-1}\tau$ contains a cycle $\pi := (c, \, d, \mathbf{y})$, with $\mathbf{y}$ empty or a sequence of distinct elements of $\{0, 1, \ldots, n-1\} \setminus \{c, d\}$. We can write $\sigma^{-1}\tau$ as $\sigma^{-1}\tau = \rho\,\omega\,\pi$ when $\omega\neq \pi$ and $\sigma^{-1}\tau = \rho\,\omega$ when $\omega=\pi$. Here $\rho \in S_n$ is disjoint from $\omega$ and $\pi$, so $\rho$ fixes both $n-1$ and $c$, and hence also $a$, $b$, and $d$.

Suppose that the conditions in the last column of line 1 of Table \ref{table:hd-conditions} hold. Then $\sigma^{-1}\tau(c) = d = c$, so that $\pi = (c) = (d) = (0)$. Since $a \neq b$, the cycle $\omega$ is nontrivial, and thus $\omega \neq \pi$.

By going through the conditions in each line of Table \ref{table:hd-conditions} in a similar manner, we can deduce all possible forms of $\omega$ and $\pi$. We summarize these in Table \ref{table:omega-pi}.

\begin{table}[ht]
\begin{center}
\begin{tabular}{rllcl}
\hline 
& $\omega$ & $\pi$ &  & $\chi$ \\
\hline \hline
\textsf{\small 1} & $(b, \, n-1, \, a, \, \mathbf{x})$ & $(0)$ & $\omega \neq \pi$ & $(b, \, c, \, a, \, \mathbf{x})$ \\
\textsf{\small 2} & $(b, \, n-1)$ & $(0)$ & $\omega \neq \pi$ & $(b, \, c)$ \\
\textsf{\small 3} & $(0)$ & $(c, \, d, \mathbf{y})$ & $\omega \neq \pi$ & $(c, \, d, \, \mathbf{y})$ \\
\textsf{\small 4} & $(0)$ & $(0)$ &  & $(0)$ \\
\textsf{\small 5} & $(0)$ & $(0)$ &  & $(0)$ \\
\hline
\textsf{\small 6} & $(b, \, n-1, \, a, \, \mathbf{x})$ & $(c, \, d, \, \mathbf{y})$ & $\omega \neq \pi$ & $(b, \, d, \, \mathbf{y}, \, c, \, a, \, \mathbf{x})$ \\
& $(b, \, n-1, \, a, \, \mathbf{x}, \, c, \, d, \, \mathbf{y})$ & $(b, \, n-1, \, a, \, \mathbf{x}, \, c, \, d, \, \mathbf{y})$ &  & $(b, \, d, \, \mathbf{y})(c, \, a, \, \mathbf{x})$ \\
\textsf{\small 7} & $(b, \, n-1)$ & $(c, \, d, \, \mathbf{y})$ & $\omega \neq \pi$ & $(b, \, d, \, \mathbf{y}, \, c)$ \\
\textsf{\small 8} & $(b, \, n-1, \, a, \, \mathbf{x})$ & $(b, \, n-1, \, a, \, \mathbf{x})$ &  & $(b, \, a, \, \mathbf{x})$ \\
\textsf{\small 9} & $(b, \, n-1, \, a, \, \mathbf{x})$ & $(b, \, n-1, \, a, \, \mathbf{x})$ &  & $(b, \, a, \, \mathbf{x})$ \\
\hline
\textsf{\small 10} & $(b, \, n-1, \, a, \, d, \, \mathbf{x})$ & $(b, \, n-1, \, a, \, d, \, \mathbf{x})$ &  & $(b, \, d, \, \mathbf{x})$ \\
\textsf{\small 11} & $(b, \, n-1, \, a, \, \mathbf{x}, c)$ & $(b, \, n-1, \, a, \, \mathbf{x}, \, c)$ &  & $(a, \, \mathbf{x}, \, c)$ \\
\textsf{\small 12} & $(b, \, n-1)$ & $(b, \, n-1)$ &  & $(0)$ \\
\textsf{\small 13} & $(b, \, n-1)$ & $(b, \, n-1)$ &  & $(0)$ \\
\hline
\textsf{\small 14} & $(b, \, n-1, \, a)$ & $(b, \, n-1, \, a)$ &  & $(0)$ \\
\hline
\end{tabular}
\caption{Factors $\omega$ and $\pi$ of $\sigma^{-1}\tau$, and $\chi$ of $\big(\sigma^{\CT}\big)^{-1}\tau^{\CT}$, corresponding to each line of Table \ref{table:hd-conditions}}
\label{table:omega-pi}
\end{center}
\end{table}

By Lemma \ref{lemma:images}, $(\sigma^{\CT})^{-1}\tau^{\CT}(x) = \sigma^{-1}\tau(x)$ for any $x \in {\rm supp}(\rho)$. Thus 
    \[ (\sigma^{\CT})^{-1}\tau^{\CT} = \rho \chi \]
for some $\chi \in S_{n-1}$ that is disjoint from $\rho$, and which moves some element in $\{a, b, c, d\}$ if $\chi\neq (0)$. 

To determine the form of $\chi$, we consider the images and inverse images of $a$, $b$, $c$, and $d$ under $\big(\sigma^{\CT}\big)^{-1}\tau^{\CT}$.
By Lemma \ref{lemma:images} we have $\big(\sigma^{\CT}\big)^{-1}\tau^{\CT}(x) = \sigma^{-1}\tau(x)$ for any $x \notin \{b,c\}$, and $\big(\tau^{\CT}\big)^{-1}\sigma^{\CT}(x) = \tau^{-1}\sigma(x)$ for any $x \notin \{a,d\}$. 

Assume that the conditions in line~1 of Table \ref{table:hd-conditions} hold. Then $a \notin \{b,c\}$, so $\big(\sigma^{\CT}\big)^{-1}\tau^{\CT}(a) = \sigma^{-1}\tau(a) = \omega(a)$ by the above. Similarly, $b \notin \{a,d\}$, so that $\big(\tau^{\CT}\big)^{-1}\sigma^{\CT}(b) = \tau^{-1}\sigma(b) = \omega^{-1}(b)$. Moreover $j\neq i$, so that $\sigma^{\CT}(j) = \sigma(j) = b$ and
    \[ \big(\sigma^{\CT}\big)^{-1}\tau^{\CT}(b) = \tau^{\CT}(j) = \tau(n-1) = d = c. \]
Since $c = \sigma(n-1) = \sigma^{\CT}(i)$ and $i\neq j$,
    \[ \big(\sigma^{\CT}\big)^{-1}\tau^{\CT}(d) = \big(\sigma^{\CT}\big)^{-1}\tau^{\CT}(c) = \tau^{\CT}(i) = \tau(i) = a.  \]
A similar analysis yields the other images and inverse images corresponding to the conditions in each line of Table \ref{table:hd-conditions}.

We now obtain $\chi$. Note that $\big(\sigma^{\CT}\big)^{-1}\tau^{\CT}(x) = \chi(x)$ for $x \in \{a, b, c, d\}$ and $\big(\sigma^{\CT}\big)^{-1}\tau^{\CT}(x) = \sigma^{-1}\tau(x)$ otherwise. Thus, for the conditions in line 1 of Table \ref{table:hd-conditions}, we have 
    \[ \chi(b) = c, \ \chi(c) = a, \ \chi(a) = \omega(a), \]
and if $\mathbf{x}$ is non-empty, $\chi(x) = \omega(x)$ for any $x \in \mathbf{x}$. Thus, $\chi = (b, \, c, \, a, \, \mathbf{x})$. The forms for $\chi$ corresponding to the rest of the lines in Table \ref{table:hd-conditions} are obtained similarly. These are summarized in the last column of Table \ref{table:omega-pi}.

The conclusions of the proposition follow immediately from Table \ref{table:omega-pi}.
\end{proof}

\section{Multiple contractions}

In this section we consider the effect on Hamming distance when applying the contraction operation several times. The main results of this section are Theorem \ref{theorem:lowerbound} and Corollaries \ref{cr1}, \ref{cr2}, \ref{cr3}, and \ref{cr4}.
These give conditions that bound the decrease in Hamming distance after repeated contractions.

For a permutation array $P \subseteq S_n$ with $\hd(P) = d$ and for any $m \in \{1, \ldots, n-1\}$ such that $\big\vert P^{\CT^m}\big\vert  > 1$, it follows from Lemma \ref{lemma:hd-bounds}.1 that
    \[ \hd\big(P^{\CT^m}\big) \geq d - 3m. \]
Indeed, for any distinct $\sigma,\tau\in P$, $\hd\big(\sigma^{\CT^i},\tau^{\CT^i}\big) - \hd\big(\sigma^{\CT^{i+1}},\tau^{\CT^{i+1}}\big) \leq 3$ for each $i \in \{0, \ldots, m-1\}$; thus
    \begin{align*}
    \hd(\sigma,\tau) - \hd\big(\sigma^{\CT^m},\tau^{\CT^m}\big)
    &= \sum_{i=0}^{m-1} \left( \hd\big(\sigma^{\CT^i},\tau^{\CT^i}\big) - \hd\big(\sigma^{\CT^{i+1}},\tau^{\CT^{i+1}}\big) \right) \\
    &\leq 3m
    \end{align*}
and so $\hd\big(\sigma^{\CT^m},\tau^{\CT^m}\big) \geq \hd(\sigma,\tau) - 3m \geq d - 3m$. In Theorem \ref{theorem:lowerbound} we give conditions that lead to an improvement of this bound. We first prove Lemma \ref{lemma:oddcycles}, which is a technical result used in the proof of Theorem \ref{theorem:lowerbound}.

\begin{lemma} \label{lemma:oddcycles}
Let $\mathcal{C} = \{\gamma_1, \ldots, \gamma_r\}$ be a set of cycles of odd length in the disjoint cycle decomposition of $\big(\sigma^{\CT}\big)^{-1}\tau^{\CT}$. Let $\omega$ and $\pi$ be the factors in the disjoint cycle decomposition of $\sigma^{-1}\tau$ as described in Proposition \ref{prop:decomposition}. Then, under a suitable rearrangement of the indices of the elements of $\mathcal{C}$, the disjoint cycle decomposition of $\sigma^{-1}\tau$ has a set $\mathcal{D}$ of cycle factors of odd length, where $\mathcal{D}$ satisfies one of the following:
    \begin{enumerate}
    \item If $\hd\big(\sigma^{\CT},\tau^{\CT}\big) = \hd(\sigma,\tau)$, then $\mathcal{D}$ is $\mathcal{C}$ or $\{\gamma_1, \ldots, \gamma_{r-1}, \omega\pi\}$, with $\vert \omega\pi\vert  = \vert \gamma_r\vert $;
    \item If $\hd\big(\sigma^{\CT},\tau^{\CT}\big) = \hd(\sigma,\tau) - 1$ and $r > 1$, then $\mathcal{D}$ is $\{\gamma_1, \ldots, \gamma_{r-2}, \delta_{r-1}\}$ for some $\delta_{r-1} \in \{\gamma_{r-1}, \omega\}$, with $\vert \delta_{r-1}\vert  \in \{\vert \gamma_{r-1}\vert , \ \vert \gamma_{r-1}\vert  + \vert \gamma_r\vert  + 1\}$;
    \item If $\hd\big(\sigma^{\CT},\tau^{\CT}\big) = \hd(\sigma,\tau) - 2$, then $\mathcal{D}$ is $\mathcal{C}$ or $\{\gamma_1, \ldots, \gamma_{r-1}, \omega\}$, with $\vert \omega\vert  = \vert \gamma_r\vert  + 2$;
    \item If $\hd\big(\sigma^{\CT},\tau^{\CT}\big) = \hd(\sigma,\tau) - 3$, then $\mathcal{D} = \mathcal{C} \cup \{\omega\}$, with $\vert \omega\vert  = 3$.
    \end{enumerate}
\end{lemma}

\begin{proof}
Recall that $\omega$ and $\pi$ are the respective cycles containing $n-1$ and $\sigma(n-1)$ in the disjoint cycle decomposition of $\sigma^{-1}\tau$. Suppose that $\sigma^{-1}\tau = \rho\,\omega\,\pi$ if $\omega\neq\pi$ and $\sigma^{-1}\tau = \rho\,\omega$ if $\omega=\pi$, where $\rho$ is disjoint from $\omega$ and $\pi$. Then by Proposition \ref{prop:decomposition}, $\big(\sigma^{\CT}\big)^{-1}\tau^{\CT} = \rho\chi$ for some permutation $\chi$ disjoint from $\rho$ which is given in Table \ref{table:omega-pi}.

Assume first that the conditions in any of the lines 1--5 and 7--14, or in line 6 with $\omega \neq \pi$, of Table \ref{table:hd-conditions} are satisfied. In each case $\chi$ is either a cycle or the identity permutation. It follows that there is at most one cycle in $\{\gamma_1, \ldots, \gamma_r\}$ which is equal to $\chi$; if $r >1$ assume without loss of generality that $\gamma_i \neq \chi$ for all $i \in \{1, \ldots, r-1\}$. Then each $\gamma_i\neq \chi$ is a factor in the disjoint cycle decomposition of $\rho$, and hence is a factor in the disjoint cycle decomposition of $\sigma^{-1}\tau$. Take $\delta_i = \gamma_i$ for each $i \in \{1, \ldots, r-1\}$. If $\hd\big(\sigma^{\CT},\tau^{\CT}\big) = \hd(\sigma,\tau) - 1$, then we have proved part of statement 2 of the proposition. Otherwise, it remains for us to choose $\delta_r$ and, if necessary, $\delta_{r+1}$. 

If $\hd\big(\sigma^{\CT},\tau^{\CT}\big) = \hd(\sigma,\tau)$, take
    \[ \delta_r = \left\{ \begin{aligned} &\gamma_r &&\text{if } \chi \neq \gamma_r, \\ & \omega\pi &&\text{if } \chi = \gamma_r. \end{aligned} \right. \]
Since $\omega\pi$ is a cycle of length $\vert \chi\vert $ by Proposition \ref{prop:decomposition} (1), in both cases $\delta_r$ is a factor in the disjoint cycle decomposition of $\sigma^{-1}\tau$ and $\vert \delta_r\vert  = \vert \gamma_r\vert $. If $\hd\big(\sigma^{\CT},\tau^{\CT}\big) = \hd(\sigma,\tau) - 2$, take
    \[ \delta_r = \left\{ \begin{aligned} &\gamma_r &&\text{if } \chi \neq \gamma_r, \\ & \omega &&\text{if } \chi = \gamma_r. \end{aligned} \right. \]
In both cases $\delta_r$ is a factor in the disjoint cycle decomposition of $\sigma^{-1}\tau$. For the case $\delta_r=\omega$, we have $\vert\omega\vert=\vert\chi\vert+2=\vert \gamma_r\vert+2$ by Proposition \ref{prop:decomposition} (3). Setting $\mathcal{D}=\{\delta_1,\ldots,\delta_r\}$ proves statements 1 and 3 of the proposition.

If $\hd\big(\sigma^{\CT},\tau^{\CT}\big) = \hd(\sigma,\tau) - 3$, then $\chi = (0)$ by Table \ref{table:omega-pi} so that $\chi \neq \gamma_r$. In this case we take $\delta_r = \gamma_r$ and $\delta_{r+1} = \omega$. Note that $\delta_r$ and $\delta_{r+1}$ are factors in the disjoint cycle decomposition of $\sigma^{-1}\tau$, and that $\omega$ is a cycle of length~$3$ by Proposition~\ref{prop:decomposition}~(4). Setting $\mathcal{D}=\{\delta_1,\ldots,\delta_{r+1}\}$ proves statement 4 of the proposition.

Now assume that the conditions in line 6 of Table \ref{table:hd-conditions} hold, with $\omega = \pi$. Then by Table \ref{table:omega-pi}, $\chi$ is a product of two disjoint cycles $\omega'$ and $\pi'$. This means that there are at most two cycles in $\{\gamma_1, \ldots, \gamma_r\}$ which are not disjoint from $\chi$. If $r >2$ assume without loss of generality that $\gamma_i \notin\{\omega', \pi'\}$ for all $i \in \{1, \ldots, r-2\}$. 
In addition, if $\{\omega',\pi'\}\neq\{\gamma_{r-1}, \gamma_r\}$ assume without loss of generality that $\gamma_{r-1}\notin \{\omega', \pi'\}$.
Then each $\gamma_i\notin\{\omega', \pi'\}$ is a factor in the disjoint cycle decomposition of $\rho$, and hence is a factor in the disjoint cycle decomposition of $\sigma^{-1}\tau$. Take $\delta_i = \gamma_i$ for each $i \in \{1, \ldots, r-2\}$, and
    \[ \delta_{r-1} = \left\{ \begin{aligned}
    &\gamma_{r-1} &&\text{if } \{\omega',\pi'\}\neq\{\gamma_{r-1}, \gamma_r\}, \\
    &\omega &&\text{if } \{\omega', \pi'\}=\{\gamma_{r-1}, \gamma_r\}.
    \end{aligned} \right. \]
In both cases $\delta_{r-1}$ is a factor in the disjoint cycle decomposition of $\sigma^{-1}\tau$. 
By Proposition \ref{prop:decomposition} (2), the permutation $\omega$ is a cycle of length $\vert \omega'\vert  + \vert \pi'\vert  + 1$ . Thus, in each case we obtain $\vert \delta_{r-1}\vert  \in \{ \vert \gamma_{r-1}\vert , \ \vert \gamma_{r-1}\vert  + \vert \gamma_r\vert  + 1 \}$. This completes the proof of statement 2 of the proposition.
\end{proof}

Theorem \ref{theorem:lowerbound} is a generalization of \cite[Theorem 3]{Sudborough_Constructing_2018} to multiple contractions. 
It may seem that the result follows inductively from \cite[Theorem 3] {Sudborough_Constructing_2018}.
On the contrary, the proof of Theorem \ref{theorem:lowerbound} relies heavily on Lemma \ref{lemma:oddcycles}. 

\begin{thm} \label{theorem:lowerbound}
Let $\sigma, \tau \in S_n$ with $\hd(\sigma,\tau) = d$. Let $m \in \{1, \ldots, n-1\}$ and assume that the disjoint cycle decomposition of $\sigma^{-1}\tau$ has no factor of odd length $\ell$ where $3 \leq \ell \leq 2m+1$. Then $\hd\big(\sigma^{\CT^m},\tau^{\CT^m}\big) \geq d - 2m$. 
\end{thm}

\begin{proof}
Since the Hamming distance between two permutations is non-negative, we assume that $d>2m$. We proceed by induction on $m$.

Suppose the disjoint cycle decomposition of $\sigma^{-1}\tau$ has no factor of length $3$. It follows from Lemma \ref{lemma:hd-bounds}.1 and Lemma \ref{lemma:hd-bounds}.2  that $\hd(\sigma,\tau) - \hd(\sigma^{\CT},\tau^{\CT}) \leq 2$. So
    \[ \hd(\sigma^{\CT},\tau^{\CT}) \geq \hd(\sigma,\tau) - 2 = d - 2, \]
and the result holds for $m = 1$.

Let $d > 2(m+1)$ and suppose that the result holds for $1 \leq m < n-1$. 
We will show that the assumption 
\[ \hd\big(\sigma^{\CT^{m+1}},\tau^{\CT^{m+1}}\big) < d - 2(m+1) = d - 2m - 2 \]
leads to the conclusion that the disjoint cycle decomposition of $\sigma^{-1}\tau$ includes a factor of odd length $\ell$ where $3 \leq \ell \leq 2m+3$.

By Lemma \ref{lemma:hd-bounds}.1 we have $\hd\big(\sigma^{\CT^m},\tau^{\CT^m}\big) - \hd\big(\sigma^{\CT^{m+1}},\tau^{\CT^{m+1}}\big) \leq 3$, so that
    \[ \hd\big(\sigma^{\CT^{m+1}},\tau^{\CT^{m+1}}\big) \geq \hd\big(\sigma^{\CT^m},\tau^{\CT^m}\big) - 3 \geq d - 2m - 3. \]
Hence we must have $\hd\big(\sigma^{\CT^{m+1}},\tau^{\CT^{m+1}}\big) = d - 2m - 3$ and $\hd\big(\sigma^{\CT^m},\tau^{\CT^m}\big) = d - 2m$.

For each $i \in \{1, \ldots, m+1\}$ let $d_i = \hd\big(\sigma^{\CT^{i-1}},\tau^{\CT^{i-1}}\big) - \hd\big(\sigma^{\CT^i},\tau^{\CT^i}\big)$. Then by the above $d_{m+1} = 3$. Again by Lemma \ref{lemma:hd-bounds}.1 we have $0 \leq d_i \leq 3$ for each $i$, and
    \[ \sum_{i=1}^{m+1} d_i = \hd(\sigma,\tau) - \hd\big(\sigma^{\CT^{m+1}},\tau^{\CT^{m+1}}\big) = d - (d - 2m - 3) = 2m + 3. \]

It follows from Lemma \ref{lemma:hd-bounds}.2 that for $d_i = 3$, the disjoint cycle decomposition of $\big(\sigma^{\CT^{i-1}}\big)^{-1}\tau^{\CT^{i-1}}$ contains a factor $\beta_i$ of length $3$. In particular, since $d_{m+1} = 3$, the cycle decomposition of $\big(\sigma^{\CT^m}\big)^{-1}\tau^{\CT^m}$ contains a $3$-cycle $\beta_m$. 
Set $E_m = \{\beta_m\}$. Applying Lemma \ref{lemma:oddcycles} to $\big(\sigma^{\CT^m}\big)^{-1}\tau^{\CT^m}$ and $\big(\sigma^{\CT^{m-1}}\big)^{-1}\tau^{\CT^{m-1}}$ with $\mathcal{C} = E_m$, for $d_m \in \{0, 2, 3\}$ we obtain a non-empty set $\mathcal{D}_{m-1}$ of odd cycles in the disjoint cycle decomposition of $\big(\sigma^{\CT^{m-1}}\big)^{-1}\tau^{\CT^{m-1}}$. Define $E_{m-1}$ by
    \[ E_{m-1} = \left\{ \begin{aligned}
    &\mathcal{D}_{m-1} &&\text{if } d_m\in\{0,2,3\}, \\
    &\varnothing &&\text{if } d_m = 1.
    \end{aligned} \right. \]
For each $i \in \{0, \ldots, m-2\}$ define $E_i$ as follows: If $E_{i+1} = \varnothing$, then
    \[ E_i = \left\{ \begin{aligned}
    &\varnothing &&\text{if $d_i\in\{0,1,2\}$}, \\
    &\{\beta_i\} &&\text{if $d_i = 3$;}
    \end{aligned} \right.\]
and if $E_{i+1} \neq \varnothing$, then
    \[ E_i = \left\{ \begin{aligned}
    &\varnothing &&\text{if $d_i = 1$ and $\vert E_{i+1}\vert  = 1$}, \\
    &\mathcal{D}_i &&\text{otherwise},
    \end{aligned} \right.\]
where $\mathcal{D}_i$ is the set obtained by applying Lemma \ref{lemma:oddcycles} to $\big(\sigma^{\CT^i}\big)^{-1}\tau^{\CT^i}$ and $\big(\sigma^{\CT^{i+1}}\big)^{-1}\tau^{\CT^{i+1}}$ with $\mathcal{C} = E_{i+1}$. Hence each non-empty $E_i$ consists of cycles of odd length which appear in the disjoint cycle decomposition of $\big(\sigma^{\CT^i}\big)^{-1}\tau^{\CT^i}$. Moreover for each $i \in \{0, \ldots, m-1\}$ we have $\vert E_i\vert  = \vert E_{i+1}\vert $ if $d_i\in\{0,2\}$, $\vert E_i\vert  = \vert E_{i+1}\vert  + 1$ if $d_i = 3$, and $\vert E_i\vert  \geq \vert E_{i+1}\vert  - 1$ if $d_i = 1$. Thus
    \[ \vert E_0\vert  \geq \vert E_m\vert  + a \cdot (-1) + (c-1) \cdot 1 = c - a, \]
where $a = \big\vert \{ i \in \{1, \ldots, m+1\}\,:\, d_i = 1 \}\big\vert $ and $c = \big\vert \{ i \in \{1, \ldots, m+1\}\,:\,d_i = 3 \}\big\vert $. 
Note that $a < c$, otherwise 
    \[ \sum_{i=1}^{m+1} d_i 
    \leq a\cdot 1 + c\cdot 3 +(m+1-a-c)\cdot 2\leq 2m + 2 < 2m + 3.\]
Hence $\vert E_0\vert  \geq 1$. That is, $E_0$ contains at least one cycle, which by definition of $E_0$ is a cycle in the disjoint cycle decomposition of $\sigma^{-1}\tau$ of odd length.

All that remains is to show that some cycle in $E_0$ has length at most $2m+3$.
In fact what we will establish is that every cycle in $E_0$ has length at most $2m+3$.
For each $i \in \{0, \ldots, m\}$, let
    \begin{align*}
    a_i &= \big\vert \{ j \in \{i+1, \ldots, m+1\}\,:\,d_j = 1 \}\big\vert , \\
    b_i &= \big\vert \{ j \in \{i+1, \ldots, m+1\}\,:\,d_j = 2 \}\big\vert , \\
    c_i &= \big\vert \{ j \in \{i+1, \ldots, m+1\}\,:\,d_j = 3 \}\big\vert .
    \end{align*}
Clearly
    \[ a_0 + 2b_0 + 3c_0 = \sum_{i=1}^{m+1} d_i = 2m + 3. \]
We claim that for each $i \in \{0, \ldots, m\}$, the sum of the lengths of the cycles in $E_i$ is at most $a_i + 2b_i + 3c_i$. Indeed, for $i = m$, we have $d_{m+1} = 3$ so that $a_m = b_m = 0$ and $c_m = 1$. Also $E_m = \{\beta_m\}$ where $\vert \beta_m\vert  = 3 = a_m + 2b_m + 3c_m$, so the claim holds for $i = m$. Now let $i \leq m$ and suppose that the claim holds for $i$. We have the following cases.

\emph{Case 1.} Assume that $E_i = \varnothing$. If $d_i \in \{0, 1, 2\}$ then $E_{i-1} = \varnothing$ so the claim holds. If $d_i = 3$ then $c_{i-1} \geq 1$ and $E_{i-1} = \{\beta_{i-1}\}$. The claim also holds because $\vert \beta_{i-1}\vert  = 3 = 3 \cdot 1 \leq a_{i-1} + 2b_{i-1} + 3c_{i-1}$.

\emph{Case 2.} Assume that $E_i \neq \varnothing$, say, $E_i = \{\gamma_1, \ldots, \gamma_r\}$. If $d_i = 1$ and $r = 1$ then $E_{i-1} = \varnothing$, so the claim holds. Otherwise $E_{i-1} = \mathcal{D}_{i-1}$; we consider each of the possibilities described in Lemma \ref{lemma:oddcycles}.

\emph{Case 2.1.} Suppose that $d_i = 0$. Then $E_{i-1} = \{\delta_1, \ldots, \delta_r\}$ with $\delta_j = \gamma_j$ for $1 \leq j \leq r-1$ and $\vert \delta_r\vert  = \vert \gamma_r\vert $. Also $a_{i-1} = a_i$, $b_{i-1} = b_i$, and $c_{i-1} = c_i$. Hence
    \[ \sum_{j=1}^r \vert \delta_j\vert  = \sum_{j=1}^r \vert \gamma_j\vert  \leq a_i + 2b_i + 3c_i = a_{i-1} + 2b_{i-1} + 3c_{i-1}. \]

\emph{Case 2.2.} Suppose that $d_i = 1$ and $r >1$. Then $E_{i-1} = \{\delta_1, \ldots, \delta_{r-1}\}$ where $\delta_j = \gamma_j$ for $1 \leq j \leq r-2$ and $\vert \delta_{r-1}\vert  \in \{ \vert \gamma_{r-1}\vert , \ \vert \gamma_{r-1}\vert  + \vert \gamma_r\vert  + 1 \}$. Also $a_{i-1} = a_i + 1$, $b_{i-1} = b_i$, and $c_{i-1} = c_i$. Hence
    \[ \sum_{j=1}^{r-1} \vert \delta_j\vert  \leq 1+\sum_{j=1}^r \vert \gamma_j\vert \leq a_{i-1} + 2b_{i-1} + 3c_{i-1}. \]

\emph{Case 2.3.} Suppose that $d_i = 2$. Then $E_{i-1} = \{\delta_1, \ldots, \delta_r\}$ with $\delta_j = \gamma_j$ for $1 \leq j \leq r-1$ and $\vert \delta_r\vert  \in \{\vert \gamma_r\vert , \ \vert \gamma_r\vert  + 2\}$. Also $a_{i-1} = a_i$, $b_{i-1} = b_i + 1$, and $c_{i-1} = c_i$. Hence
    \[ \sum_{j=1}^r \vert \delta_j\vert  \leq 2+\sum_{j=1}^r \vert \gamma_j\vert\leq a_{i-1} + 2b_{i-1} + 3c_{i-1}. \]

\emph{Case 2.4.} Suppose that $d_i = 3$. Then $E_{i-1} = \{\delta_1, \ldots, \delta_{r+1}\}$ with $\delta_j = \gamma_j$ for $1 \leq j \leq r$ and $\vert \delta_{r+1}\vert  = 3$. Also $a_{i-1} = a_i$, $b_{i-1} = b_i$, and $c_{i-1} = c_i + 1$. Hence
    \[ \sum_{j=1}^{r+1} \vert \delta_j\vert  = 3+\sum_{j=1}^r \vert \gamma_r\vert\leq a_{i-1} + 2b_{i-1} + 3c_{i-1}. \]
In each case the sum of the lengths of the cycles in $E_{i-1}$ is bounded above by $a_{i-1} + 2b_{i-1} + 3c_{i-1}$. This proves the claim.

In particular the sum $L$ of the lengths of the cycles in $E_0$ satisfies
    \[  L \leq a_0 + 2b_0 + 3c_0=2m + 3.\]
Hence each cycle in the non-empty set $E_0$ has length at most $2m + 3$. 
Therefore, each element of $E_0$ is a cycle in the disjoint cycle decomposition of $\sigma^{-1}\tau$ of odd length $\ell$ where $3 \leq \ell \leq 2m+3$. 
This completes the proof.
\end{proof}

We now apply Theorem \ref{theorem:lowerbound} to obtain a lower bound for the Hamming distance of a permutation array after applying the contraction operation $m$ times.

\begin{cor} \label{cr1}
Let $P \subseteq S_n$ be a permutation array with Hamming distance $d$. Let $m \in \{1, \ldots, n-1\}$, and suppose that $P$ has the property that for any $\sigma,\tau \in P$ the disjoint cycle decomposition of $\sigma^{-1}\tau$ has no cycle of odd length $\ell \in \{3, \ldots, 2m+1\}$. Then
    \begin{enumerate}
    \item $\hd\big(P^{\CT^m}\big) \geq d-2m$, and
    \item if $d > 2m$ then $\big\vert P^{\CT^m}\big\vert  = \vert P\vert $.
    \end{enumerate}
\end{cor}

\begin{proof}
Note that $\hd(\sigma,\tau) \geq d$ for all distinct $\sigma,\tau \in P$, so by Theorem \ref{theorem:lowerbound},
    \[ \hd\big(\sigma^{\CT^m},\tau^{\CT^m}\big) \geq \hd(\sigma,\tau) - 2m \geq d - 2m. \]
Since $\sigma$ and $\tau$ are arbitrary, it follows that $\hd(P^{\CT^m}) \geq d - 2m$. If $d - 2m > 0$ then $\hd\big(\sigma^{\CT^m},\tau^{\CT^m}\big) > 0$ for all distinct $\sigma, \tau \in P$. Hence $\sigma^{\CT^m} \neq \tau^{\CT^m}$ for all distinct $\sigma, \tau \in P$, and therefore $\big\vert P^{\CT^m}\big\vert  = \vert P\vert $. This completes the proof.
\end{proof}

The remaining results arise from the special cases where $P$ is one of the groups $\AGL(1,q)$, $\PGL(2,q)$, $\AGamL(1,q)$, or $\PGamL(2,q)$. 
They are generalizations of \cite[Theorem 4 and Corollaries 1, 3, and 4]{Sudborough_Constructing_2018} for multiple contractions.

\begin{cor} \label{cr2}
Let $q$ be a prime power and $m$ a positive integer such that $q > 2m + 1$ and $q \not\equiv 0,1 \imod{2i+1}$ for any $i \in \{1, \ldots, m\}$. Then $M(q-m, q-2m-1) \geq q(q-1)$.
\end{cor}

\begin{proof}
The hypothesis implies that $\vert \AGL(1,q)\vert  = q(q-1)$ has no odd divisor between $3$ and $2m+1$. Hence, for any distinct $\sigma,\tau \in \AGL(1,q)$, $\big\vert \sigma^{-1}\tau\big\vert $ has no odd divisor between $3$ and $2m+1$. It follows that $\sigma^{-1}\tau$ has no cycle of odd length between $3$ and $2m+1$ in its disjoint cycle decomposition. By \cite[Theorem 1 (ii)]{Deza_Max_Number_Permutations_1977} we have $\hd(\AGL(1,q)) \geq q-1$, so by Corollary \ref{cr1}.1 we have $\hd\big(\AGL(1,q)^{\CT^m}\big) \geq q - 1 - 2m$. In addition, since $q>2m+1$, we get from Corollary \ref{cr1}.2 that $\big\vert \AGL(1,q)^{\CT^m}\big\vert  = \vert \AGL(1,q)\vert  = q(q-1)$. Note that $\AGL(1,q)^{\CT^m} \subseteq S_{q-m}$. The result follows immediately.
\end{proof}

\begin{cor} \label{cr3}
Let $q$ be a prime power such that $q > 5$, $q \equiv 2 \imod{3}$, and $q \not \equiv 0,1 \imod{5}$. Then $M(q-1, q-5) \geq q(q-1)(q+1)$ and $M(q-2,q-5)\geq q(q+1)$.
\end{cor}

\begin{proof}
Since $\PGL(2,q)$ has a sharply $3$-transitive action on a set of $q+1$ elements, it follows from \cite[Theorem 1 (iii)]{Deza_Max_Number_Permutations_1977} that $\hd(\PGL(2,q)) \geq (q + 1) - 3 + 1 = q - 1$. 

Let $\sigma,\tau$ be distinct elements of $\PGL(2,q)$. If $\hd(\sigma,\tau) = q + 1$, then by Lemma~\ref{lemma:hd-bounds}.1 we have
    \[ \hd\big(\sigma^{\CT^2},\tau^{\CT^2}\big) \geq [(q + 1)-3]-3 = q - 5. \]
Suppose $\hd(\sigma,\tau) \neq q + 1$. Then $\hd(\sigma,\tau) \geq \hd(\PGL(2,q)) = q - 1$. Also, since $\PGL(2,q) \subset S_{q + 1}$ and $\hd(\sigma,\tau) \neq q + 1$, $\sigma$ and $\tau$ must fix a common point $x$ and $\sigma^{-1}\tau$ is in the stabilizer of $x$ in $\PGL(2,q)$. Any point stabilizer in $\PGL(2,q)$ is isomorphic to $\AGL(1,q)$. It follows from the hypotheses that $\vert \AGL(1,q)\vert  = q(q-1)$ is not divisible by $3$ or $5$, so the disjoint cycle decomposition of $\sigma^{-1}\tau$ has no factor of length $3$ or $5$. Applying Theorem \ref{theorem:lowerbound} we obtain
    \[ \hd\big(\sigma^{\CT^2},\tau^{\CT^2}\big) \geq q - 1 - 2 \cdot 2 = q - 5. \]
Finally, since $q - 5 > 0$, we have $\big\vert \PGL(2,q)^{\CT^2}\big\vert  = \vert \PGL(2,q)\vert  = q(q-1)(q+1)$ by Corollary \ref{cr1}.2.
This proves that $M(q-1, q-5) \geq q(q-1)(q+1)$.

Since $M(n,d)\leq nM(n-1,d)$ \cite{Chu_Codes_powerline_2004}, it follows that $M(q-2,q-5)\geq M(q-1,q-5)/(q-1)\geq q(q+1)$.
\end{proof}
Corollary 5.5 also follows from \cite[Theorem 3(b)]{Sudborough_Constructing_2018}. 
Observe that the lower bound $q(q+1)$ for $M(q-2,q-5)$ is slightly better than the lower bound $q(q-1)$ in \cite[Corollary 4(ii)]{Sudborough_Constructing_2018}.
Note though that the proof of Corollary \ref{cr3} does not extend to $m>3$ contractions, for then $q+1-3m<q-1-2m$.

The following corollary is proven in a manner similar to Corollaries \ref{cr2} and \ref{cr3}. It is obtained from the application of multiple contractions to $\AGamL(1,q)$ and $\PGamL(2,q)$, whose sizes are $k q(q-1)$ and $kq(q-1)(q+1)$, respectively. The Hamming distances of the permutation arrays from these groups are $q-p^{k^*}$, where $k^*$ is the largest proper divisor of $k$ \cite[Theorems 1 and 2]{Sudborough_Constructing_2018}. 

\begin{cor} \label{cr4}
Let $q = p^k$ with $p$ prime and $k \geq 2$, and let $k^*$ be the largest proper divisor of $k$. Let $m$ be a positive integer such that $q - p^{k^*} - 2m > 0$ and $kq(q-1) \not\equiv 0 \imod{2i+1}$ for any $i \in \{1, \ldots, m\}$. The following statements hold.

\begin{enumerate}
    \item $M(q - m, q - p^{k^*} - 2m) \geq kq(q-1)$.
    \item $M(q + 1 - m, q - p^{k^*} - 2m) \geq kq(q-1)(q+1)$, if $q + 1-3m \geq q-p^{k^*}-2m$.
\end{enumerate}

\end{cor}

We remark that Theorem \ref{theorem:lowerbound} cannot be applied to the sporadic multiply-transitive Mathieu groups $M(11)$, $M(12)$, $M(22)$, $M(23)$, and $M(24)$ because each of their orders is divisible by 3.

\section{Analysis of Lower Bounds Obtained After Multiple Contractions}
\captionsetup[table]{list=yes}

We now inspect the lower bounds that we obtained in Section 5.
Known lower bounds can be found from tables and formulae in  \cite{Chu_Codes_powerline_2004,Smith_table_permutation_codes_2012,Sudborough_Constructing_2018,Bereg2021}.  
Unfortunately, the lower bounds for $M(q-m,q-2m-1)$ obtained from Corollary \ref{cr2} were not competitive with known lower bounds.
Similarly, lower bounds for $M(q - m, q - p^{k^*} - 2m)$ and $M(q + 1 - m, q - p^{k^*} - 2m)$ obtained from Corollary \ref{cr4} were not better than the lower bounds in~\cite{Bereg2021} or \cite{Chu_Codes_powerline_2004}.

Nevertheless, we obtain new lower bounds for $M(q-1,q-5)$ for suitable prime powers $q\leq 149$ (Table \ref{tab:Table 0.5}).
If a lower bound for $M(q-1,q-5)$ is not present in the literature, we compared our obtained lower bound instead with known lower bounds for $M(q,q-5)/q$ (see \cite{Chu_Codes_powerline_2004}).
The largest comparable lower bound is included in Table \ref{tab:Table 0.5}. 
Lower bounds in bold are improvements of lower bounds found in the literature. 

The lower bounds in Table \ref{tab:Table 0.5} were verified by applying two contractions to the  
corresponding projective general linear groups via GAP 4.11.1. 
The permutation representations of these groups were obtained through GAP's ``FinInG" package and the Hamming distances were obtained through the GAP command ``DistancePerms()". Moreover, the contraction of a permutation $\sigma$ was accomplished in GAP by multiplying $\sigma$ by the cycle $(n-1,\sigma(n-1))$ from the right.

\begin{table}[ht]
\centering
\begin{adjustbox}{width=.75\textwidth}
\begin{tabular}{|c|c|c|c|c|}
\hline
$q$ & $M(q-1,q-5)$ & Obtained Lower Bound &
Previous Lower Bound & 
Reference for Previous Lower Bound \\ \hline
 8  &   {$M(7,3)$}  &   504          &   2520  &    \cite[Proposition 1.1]{Chu_Codes_powerline_2004}            \\
 17   & { $M(16,12)$ }  &   4896          & 40320   & \cite[Table 3]{Smith_table_permutation_codes_2012}                     \\
  23  &  $M(22,18)$    &    12144         &13680     &     \cite[Table 2]{Sudborough_Constructing_2018}       \\ 
29       &   {$M(28,24)$ }  &  24360           &58968     &       \cite[Corollary 3]{Sudborough_Constructing_2018}                  \\
 32   & {$M(31,27)$ }    &  32736           &$43400\leq M(32,27)/32$    & \cite[Table 3]{Bereg2021}                       \\
  47  &    {$M(46,42)$} &     103776        & $103822\leq M(47,42)/47$   &    \cite[Table 3]{Bereg2021}           \\
53       &    {$M(52,48)$}  &  148824           &$148828\leq M(53,48)/53$     &  \cite[Table 3]{Bereg2021}                    \\
  59  &\bm{$M(58,54)$}  &   205320          & $102834\leq M(59,54)/59$   & \cite[Table 3]{Bereg2021}                \\
83    & {$M(82,78)$ }    &  571704           & $576605\leq M(83,78)/83$    &  \cite[Table 3]{Bereg2021}              \\
 89      &  \bm{$M(88,84)$}    & 704880            &$352704\leq M(89,84)/89$    & \cite[Table 3]{Bereg2021}                      \\
 107   &\bm{$M(106,102)$}     &   1224936          &  $1213592\leq M(107,102)/107$   & \cite[Table 3]{Bereg2021}             \\ 
 113   &\bm{$M(112,108)$}     &   1442784          &  $1430128\leq M(113,108)/113$   & \cite[Table 3]{Bereg2021}             \\
 128   &\bm{$M(127,123)$}     &   2097024          &  $1032510\leq M(128,123)/128$   & \cite[Table 3]{Bereg2021}             \\
 137   &\bm{$M(136,132)$}     &   2571216          &  $2552584\leq M(137,132)/137$   & \cite[Table 3]{Bereg2021}             \\
 149   &\bm{$M(148,144)$}     &   3307800          &  $1632144\leq M(149,144)/149$   & \cite[Table 3]{Bereg2021}             \\
    \hline           
\end{tabular}
\end{adjustbox}
\caption{Lower bounds for $M(n,d)$ by two contractions of $\PGL(2,q)$}
\label{tab:Table 0.5}
\end{table}

\section*{Acknowledgements}
The authors thank Dr.~I.H.~Sudborough for suggesting this problem on the Hamming distances of multiply contracted permutation arrays.

\bibliographystyle{amsplain}
\bibliography{Amarra_Loquias_Briones_revised_20Mar22}

\end{document}